\documentclass{amsart}
\usepackage{amsmath,dsfont,amsfonts}
\usepackage{color}

\newtheorem{teo}{Theorem}[section]
\newtheorem{prop}[teo]{Proposition}
\newtheorem{lm}[teo]{Lemma}
\newtheorem{coro}[teo]{Corollary}
\newtheorem{rem}[teo]{Remark}

\newcommand{\RR}{{\mathbb{R}}}

\newcommand{\inspol}{{\mathcal{A}}}

\newcommand{\der}{\partial}
\newcommand{\ep}{\varepsilon}

\newcommand{\om}{\Omega}
\newcommand{\dive}{\text{\normalfont div}}

\def\qed{\hfill$\square$\vspace{0.5cm}}    

\numberwithin{equation}{section}

\begin{document}

\title[Global Lipschitz stability for polygonal inclusions]{Global Lipschitz stability estimates for polygonal conductivity inclusions from boundary measurements }

\author[E.~Beretta et al.]{Elena~Beretta}
\address{Dipartimento di Matematica ``Brioschi'',
Politecnico di Milano \& New York University Abu Dhabi}
\email{eb147@nyu.edu}
\author[]{Elisa~Francini}
\address{Dipartimento di Matematica e Informatica ``U. Dini'',
Universit\`{a} di Firenze}
\email{elisa.francini@unifi.it}


\keywords{polygonal inclusions, conductivity equation, stability, inverse problems}

\subjclass[2010]{35R30, 35J25}
\begin{abstract}
We derive Lipschitz stability estimates for the Hausdorff distance of polygonal conductivity inclusions in terms of the Dirichlet-to-Neumann map.
\end{abstract}

\maketitle

%
%
\section{Introduction}
In this paper we establish Lipschitz stability estimates  for a certain class of discontinuous conductivities $\gamma$  in terms of the Dirichlet-to-Neumann map.\\
More precisely,  we consider the following boundary value problem
\begin{equation}\label{conductivity}
    \left\{\begin{array}{rcl}
             \textrm{ div }((1+(k-1)\chi_{\mathcal{P}})\nabla u) & = & 0\mbox{ in }\om\subset\mathbb{R}^2, \\
             u & = & \phi \mbox{ on }\der\om,
           \end{array}
    \right.
\end{equation}
where $\phi\in H^{1/2}\left(\der\om\right)$,  $\mathcal{P}$ is a polygonal inclusion strictly contained in a planar, bounded domain $\Omega$ and $k\neq 1$ is a given, positive constant. \\Our goal is  to determine the polygon $\mathcal{P}$ from the knowledge of the Dirichlet-to-Neumann map
\begin{equation*}
    \Lambda_{\gamma}: H^{1/2}\left(\der\om\right)\to H^{-1/2}\left(\der\om\right)
\end{equation*}
with
\begin{equation*}
    \Lambda_{\gamma}(f):=\gamma{\frac{\der u}{\der \nu}}\in  H^{-1/2}\left(\der\om\right).
\end{equation*}
This class of conductivity inclusions is quite common in applications, like for example in geophysics exploration, where the medium (the earth) under inspection contains heterogeneities in the form of rough bounded subregions (for example subsurface salt bodies) with different conductivity properties \cite{ZK}.

Moreover, polygonal inclusions represent a class in which Lipschitz stable reconstruction from boundary data can be expected \cite{BdHFV}.
In fact, it is well known that the determination of an arbitrary (smooth) conductivity inclusion from the Dirichlet-to-Neumann map is  exponentially ill-posed \cite{DiCR}. On the other hand, restricting the class of admissible inclusions to a compact subset of a finite dimensional space regularizes the inverse problem and allows to establish Lipschitz stability estimates and stable reconstructions (see \cite{BV},\cite{BMPS}, \cite{AS}, \cite{H}).
In order to show our main result we follow a similar approach as the one in \cite{BdHFV} and take advantage of a recent result obtained by the authors in \cite{BFV17} where they prove Fr\'echet differentiability of the  Dirichlet-to-Neumann map with respect to affine movements of vertices of polygons and where they establish an explicit representation formula for the derivative.\\
We would like to mention that our result relies on the knowledge of infinitely many measurements though one expects that finitely many measurements should be enough to determine a polygonal inclusion. In fact, in \cite{BFS} the authors show that if the inclusion is a convex polyhedron, then one suitably assigned current at the boundary of the domain $\Omega$  and the corresponding measured boundary potential  are enough to uniquely determine  the inclusion (see also \cite{S} for the unique determination of an arbitrary polygon from two appropriately chosen pairs of boundary currents and potentials and also \cite{KY} where a convex polygon is uniquely determined in the case of variable conductivities). Unfortunately in the aforementioned papers, the choice of the current fields is quite special and the proof of uniqueness is not constructive. In fact, to our knowledge, no stability result for polygons from few boundary measurements has been derived except for the local stability result obtained in \cite{BFI}. On the other hand, in several applications, like the geophysical one, many measurements are at disposal  justifying the use of the full Dirichlet-to-Neumann map, \cite{BCFLM}.\\
The paper is organized as follows: in Section 2 we state our main assumptions and the main stability result. Section 3 is devoted to the proof of our main result and finally,  Section 4 is devoted to concluding remarks about the results and possible extensions.

\section{Assumptions and main result}
Let $\om\subset\RR^2$ be a bounded open set with $diam(\om)\leq L$.
We denote either by $x=(x_1,x_2)$ and by $P$ a point in $\RR^2$.
We assume that $\der\om$ is of Lipschitz class with constants $r_0$ and $K_0>1$ that means that for every point $P$ in $\der\om$ there exists a coordinate system such that $P=0$ and
\[\om\cap \left([-r_0,r_0]\times[-K_0r_0,K_0r_0]\right)=\left\{(x_1,x_2)\,:\, x_1\in[-r_0,r_0], x_2>\phi(x_1)\right\}\]
for a Lipschitz continuous function $\phi$ with Lipschitz norm smaller than $K_0$.

We denote by $dist(\cdot,\cdot)$ the euclidian distance between points or subsets in $\RR^2$. Later on we will also define the Haussdorff distance $d_H(\cdot,\cdot)$.

Let $\inspol$ the set of closed, simply connected, simple polygons $\mathcal{P}\subset \om$ such that:
\begin{equation}\label{lati}
\mathcal{P}\mbox{ has at most }N_0\mbox{ sides each one with length greater than }d_0;
\end{equation}
\begin{equation}\label{lip}\der\mathcal{P}\mbox{ is of Lipschitz class with constants }r_0\mbox{ and }K_0,\end{equation}
there exists a  constant $\beta_0\in (0,\pi/2]$ such that the angle $\beta$ in each vertex of $\mathcal{P}$ satisfies the conditions
\begin{equation}\label{angoli}
\beta_0\leq\beta\leq 2\pi-\beta_0\mbox{ and } |\beta-\pi|\geq\beta_0,
\end{equation}
and
\begin{equation}\label{distanza}
dist(\mathcal{P},\der\om)\geq d_0.
\end{equation}

Notice that we do not assume convexity of the polygon.

Let us consider the problem
\begin{equation*}
    \left\{\begin{array}{rcl}
             \dive(\gamma\nabla u) & = & 0\mbox{ in }\om, \\
             u&=&\phi\mbox{ on }\der\om,\\
           \end{array}
    \right.
\end{equation*}
where $\phi\in H^{1/2}(\der\om)$ and
\begin{equation}\label{gamma}
\gamma=1+(k-1)\chi_{\mathcal{P}},
\end{equation}
for a given $k>0$, $k\neq 1$ and for $\mathcal{P}\in\inspol$. 
The constants $k$, $r_0$, $K_0$, $L$, $d_0$, $N_0$ and $\beta_0$ will be referred to as the \textit{a priori data}.\\ In the sequel  we will introduce a number of constants depending only on the  \textit{a priori data} that we will always denote by $C$. The values of these constants might differ from one line to the other.

Let us consider the Dirichlet to Neumann map
\[\begin{array}{rcl}\Lambda_\gamma: H^{1/2}(\der\om)&\to& H^{-1/2}(\der\om)\\
\phi&\to&\gamma{\frac{\der u}{\der n}}_{|_{\der\om}},\end{array}\]
whose norm in the space of linear operators $\mathcal{L}(H^{1/2}(\der\om), H^{-1/2}(\der\om))$ is defined by
\[\|\Lambda_\gamma\|_*=\sup\left\{\|\Lambda_\gamma\phi\|_{H^{-1/2}(\der\om)}/\|\phi\|_{H^{1/2}(\der\om)}\,:\,\phi\neq 0\right\}. \]
\begin{teo}\label{mainteo}
Let $\mathcal{P}^0,\mathcal{P}^1\in\inspol$ and let
\[\gamma_0=1+(k-1)\chi_{\mathcal{P}^0}\mbox{ and }\gamma_1=1+(k-1)\chi_{\mathcal{P}^1}.\]
There exist $\ep_0$ and $C$ depending only on the a priori data such that, if
\[\|\Lambda_{\gamma_0}-\Lambda_{\gamma_1}\|_*\leq \ep_0,\]
then $\mathcal{P}^0$ and $\mathcal{P}^1$ have the same number $N$ of vertices $\left\{P_j^0\right\}_{j=1}^N$ and $\left\{P_j^1\right\}_{j=1}^N$ respectively. Moreover, 
\begin{equation}\label{stab1}
d_{H}\left( \partial \mathcal{P}^{0},\partial \mathcal{P}^{1}\right)\leq C
\|\Lambda_{\gamma_0}-\Lambda_{\gamma_1}\|_*\quad \mbox{ for every }j=1,\ldots,N.
\end{equation}
\end{teo}
\vskip 8truemm
\begin{rem}
Observe that our stability estimate is a global one. 
In fact, if $\left \Vert \Lambda _{\gamma_0}-\Lambda _{\gamma_1}\right \Vert
_{\ast }>\varepsilon _{0}$,  since the
following trivial inequality  holds 
\[
d_{H}\left( \partial \mathcal{P}^{0},\partial \mathcal{P}^{1}\right) \leq 2L,
\]
we have trivially%
\begin{equation}
d_{H}\left( \partial \mathcal{P}^{0},\partial \mathcal{P}^{1}\right) \leq
2L\leq 2L\frac{\left \Vert \Lambda _{0}-\Lambda _{1}\right \Vert _{\ast }}{%
\varepsilon _{0}}  \label{stab2}
\end{equation}%
Therefore, in \textit{any case, }by\textit{\ (\ref{stab1}), (\ref{stab2})} we obtain 
the global estimate
\[
d_{H}\left( \partial \mathcal{P}^{0},\partial \mathcal{P}^{1}\right) \leq
\left( C+\frac{2L}{\varepsilon _{0}}\right) \left \Vert \Lambda _{0}-\Lambda
_{1}\right \Vert _{\ast }.
\]

\end{rem}
\section{Proof of the main result}
The proof of Theorem \ref{mainteo} follows partially the strategy used in \cite{BdHFV} in the case of the Helmholtz equation.

The first step of the proof is a rough stability estimate for $\|\gamma_0-\gamma_1\|_{L^2(\om)}$ which is stated in Section \ref{srozza} and which follows from a result by Clop, Faraco and Ruiz \cite{CFR}. Then, in section \ref{sgeo}, we show a rough stability estimate for the Hausdorff distance of the polygons. We also show that if $\|\Lambda_{\gamma_0}-\Lambda_{\gamma_1}\|_*$ is small enough, then the two polygons have the same number of vertices and that the distance from vertices of $\mathcal{P}^0$ and vertices of $\mathcal{P}^1$ is small. For this reason it is possible to define a coefficient $\gamma_t$ that goes smoothly from $\gamma_0$ to $\gamma_1$ and the corresponding Dirichlet to Neumann map.
We prove that the Dirichlet to Neumann map is differentiable (section \ref{sFdiff}), its derivative is continuous (section \ref{sderivcont}) and bounded from below (section \ref{sboundbelow}). These results finally give the Lipschitz stability estimate of Theorem \ref{mainteo}.
\subsection{A logarithmic stability estimate}\label{srozza}
As in \cite{BdHFV}, we can show that, thanks to Lemma 2.2 in \cite{MP}  there exists a constant $\Gamma_0$, depending only on the a priori data, such that, for $i=0,1$,
\begin{equation}\label{hs}
\|\gamma_i\|_{H^s(\om)}\leq \Gamma_0\quad\forall s\in(0,1/2).
\end{equation}
Due to this regularity of the coefficients, we can apply Theorem 1.1 in \cite{CFR} and obtain the following logarithmic stability estimate:
\begin{prop}\label{strozza} There exist $\alpha<1/2$ and $C>1$, depending only on the a priori data, such that
\begin{equation}\label{stimarozzal2}\|\gamma_1-\gamma_0\|_{L^2(\om)}\leq C\left|\log\|\Lambda_{\gamma_0}-\Lambda_{\gamma_1}\|_*\right|^{-\alpha^2/C},\end{equation}
if $\|\Lambda_{\gamma_0}-\Lambda_{\gamma_1}\|_*<1/2$.
\end{prop}
\subsection{A logarithmic stability estimate on distance of vertices}\label{sgeo}
In this section we want to show that, due to the assumptions on polygons in $\inspol$, estimate \eqref{stimarozzal2} yields an estimate on the Hausdorff distance $d_H(\der\mathcal{P}^0,\der\mathcal{P}^1)$ and, as a consequence, on the distance of the vertices of the polygons.

It is immediate to get from \eqref{stimarozzal2} that
\begin{equation}\label{diffsimm}
	\left|\mathcal{P}^0\Delta\mathcal{P}^1\right|\leq \frac{C}{|k-1|}\left|\log\|\Lambda_{\gamma_0}-\Lambda_{\gamma_1}\|_*\right|^{-\alpha^2/C}\end{equation}
Now,
we show that \eqref{diffsimm} implies an estimate on the Hausdorff distance of the boundaries of the polygons.

Let us recall the definition of the Hausdorff distance between two sets $A$ and $B$:
\[d_H(A,B)=max\{\sup_{x\in A}\inf_{y\in B}dist(x,y),\sup_{y\in B}\inf_{x\in A}dist(y,x)\}\]

The following result holds:
 \begin{lm}\label{dadiffahaus}
Given two polygons $\mathcal{P}^0$ and $\mathcal{P}^1$ in $\inspol$, we have
\[d_H(\der\mathcal{P}^0,\der\mathcal{P}^1)\leq C\sqrt{\left|\mathcal{P}^0\Delta\mathcal{P}^1\right|}\]
where $C$ depends only on the a priori data.
\end{lm}
\begin{proof}
Let $d=d_H(\der\mathcal{P}^0,\der\mathcal{P}^1)$. Assume $d>0$ (otherwise the thesis is trivial) and let $x_0\in\der \mathcal{P}^0$ such that $dist(x_0,\der\mathcal{P}^1)=d$. Then,
\[B_d(x_0)\subset \RR^2\setminus\der \mathcal{P}^1.\]
There are two possibilities:\\
(i) $B_d(x_0)\subset \RR^2\setminus\mathcal{P}^1$ or \\
(ii) $B_d(x_0)\subset \mathcal{P}^1$.

In case (i),
$B_d(x_0)\cap \mathcal{P}^0\subset \mathcal{P}^0\setminus\mathcal{P}^1$.
The definition of $\inspol$ implies that, if $d\leq d_0$, there is a constant $C>1$ depending only on the a priori data such that
\[\left|B_d(x_0)\cap \mathcal{P}^0\right|\geq \frac{d^2}{C^2}.\]
If $d\geq d_0$ we trivially have
\[\left|B_d(x_0)\cap \mathcal{P}^0\right|\geq \left|B_{d_0}(x_0)\cap \mathcal{P}^0\right|\geq \frac{d_0^2}{C^2},\]
hence, in any case, for
\[f(d)=\left\{\begin{array}{rl}d^2/C^2&\mbox{ if }d<d_0\\
d_0^2/C^2&\mbox{ if }d\geq d_0\end{array}\right.\]
we have
\[f(d)\leq \left|B_d(x_0)\cap \mathcal{P}^0\right|\leq \left|\mathcal{P}^0\Delta\mathcal{P}^1\right|
.\]
Now, if $\left|\mathcal{P}^0\Delta\mathcal{P}^1\right|<\frac{d_0^2}{C^2}$, then $f(d)=\frac{d^2}{C^2}\leq \left|\mathcal{P}^0\Delta\mathcal{P}^1\right|$ gives $d\leq C\sqrt{\left|\mathcal{P}^0\Delta\mathcal{P}^1\right|}$.
On the other hand, if $\left|\mathcal{P}^0\Delta\mathcal{P}^1\right|\geq \frac{d_0^2}{C^2}$ we have
\[\frac{d^2}{C^2}\leq \frac{L^2}{C^2}\leq \frac{L^2}{C^2}\frac{\left|\mathcal{P}^0\Delta\mathcal{P}^1\right|}{d_0^2/C^2}\]
that gives $d\leq \frac{LC}{d_0}\sqrt{\left|\mathcal{P}^0\Delta\mathcal{P}^1\right|}$.

In case (ii),  $B_d(x_0)\subset \mathcal{P}^1$, hence
\[B_d(x_0)\setminus \mathcal{P}^0\subset \mathcal{P}^1\setminus\mathcal{P}^0\subset\mathcal{P}^1\Delta\mathcal{P}^0.\]
Proceeding as above we have
\[f(d)\leq \left|B_d(x_0)\setminus \mathcal{P}^0\right|\leq \left|\mathcal{P}^0\Delta\mathcal{P}^1\right|\]
and the same conclusion follows.
\end{proof}

\begin{prop}\label{propvertici}
Given the set of polygons $\inspol$ there exist $\delta_0$ and $C$ depending only on the a priori data such that, if for some $\mathcal{P}^0$, $\mathcal{P}^1\in \inspol$ we have
\[d_H(\der\mathcal{P}^0,\der\mathcal{P}^1)\leq \delta_0,\]
then
 $\mathcal{P}^0$ and $\mathcal{P}^1$ have the same number $N$ of vertices $\{P^0_i\}_{i=1}^N$ and $\{P^1_i\}_{i=1}^N$, respectively, that can be ordered in such a way that
\[dist(P^0_i,P^1_i)\leq Cd_H(\der\mathcal{P}^0,\der\mathcal{P}^1) \mbox{ for every }i=1,\ldots,N.\]
\end{prop}
\begin{proof}
Let us denote by
\[\delta=d_H(\der\mathcal{P}^0,\der\mathcal{P}^1).\]

Assume $\mathcal{P}^0$ has $N$ vertices and that $\mathcal{P}^1$ has $M$ vertices. We now will show that for any vertex $P^0_i\in \der\mathcal{P}^0$ there exists a vertex $P^1_j\in\der\mathcal{P}^1$ such that
$dist(P^0_i,P^1_j)<C\delta$. By assumption \eqref{lati} this implies that $N\leq M$. Interchanging the role of $\mathcal{P}^0$  and $\mathcal{P}^1$ we get that $M\leq N$ which implies that $M=N$.

Let $P$ be one of the vertices in $\der\mathcal{P}^0$ and let us consider the side $l^\prime$ of $\der\mathcal{P}^1$ that is close to $P$. Let us set the coordinate system with origin in the midpoint of $l^\prime$ and let $(\pm l/2,0)$ be the endpoint of $l^\prime$.

By definition of the Hausdorff distance, $P\in \mathcal{U}_\delta=\left\{x\in\RR^2\,:\,dist(x,l^\prime)\leq\delta\right\}$.

Now we want to show that, due to the assumptions on $\inspol$, for sufficiently small $\delta$ there is a constant $C$ such that the distance between $P$ and one of the endpoints of $l^\prime$ is smaller than $C\delta$. The reason is that if $P$ is too far from the endpoints, assumption \eqref{angoli} on $\mathcal{P}^0$ cannot be true.

Let us choose $\delta$ small enough to have:
\begin{equation}\label{cond1}
	\delta<K_0 r_0
\end{equation}
(this guarantees that the $\delta$-neighborhood of each side of $\mathcal{P}^1$ does not intersect the $\delta$-neighborhood of a non adjacent side), and
\begin{equation}\label{cond2}
	\delta<\frac{d_0\sin\beta_0}{16}.
\end{equation}

Notice that, by assumption \eqref{angoli} and by \eqref{cond1}, the rectangle
\[R=\left[-\frac{l}{2}+\frac{2\delta}{\sin\beta_0},\frac{l}{2}-\frac{2\delta}{\sin\beta_0}\right]\times[-\delta,\delta]\]
does not intersect the $\delta$-neighborhood of any other side of $\mathcal{P}^1$.

Let us now show that $P$ cannot be contained in a slightly smaller rectangle
\[R^\prime=\left[-\frac{l}{2}+\lambda,\frac{l}{2}-\lambda\right]\times[-\delta,\delta],\]
where $\lambda=\frac{6\delta}{\sin\beta_0}$.

Let us assume by contradiction that $P\in R^\prime$ and consider the two sides of $\der\mathcal{P}^0$ with an endpoint at $P$. These sides have length greater than $d_0$, hence they intersect $\der B_{\lambda/2}(P)$ in two points $Q_1$ and $Q_2$ in $R$ (because $\lambda/2<\lambda-\frac{2\delta}{\sin\beta_0}$).

Since $\lambda/2>2\delta$ the intersection $\der B_{\lambda/2}(P)\cap R$ is the union of two disjoint arcs. We estimate the angle of $\mathcal{P}^0$ at $P$ in the two alternative cases:\\
(i) $Q_1$ and $Q_2$ are on the same arc or\\
(ii)  $Q_1$ and $Q_2$ are on different arcs.

In case (i), the angle at $P$ is smaller than $\arcsin\left(\frac{4\delta}{\lambda}\right)$
(the angle is smaller than $\arcsin\left(\frac{2(\delta-b)}{\lambda}\right)+\arcsin\left(\frac{2(\delta+b)}{\lambda}\right)$, where $b$ is the $y$-coordinate of $P$,  that is maximum for $b=\pm\delta$).

In order for \eqref{angoli} to be true we should have
\[\arcsin\left(\frac{4\delta}{\lambda}\right)=\arcsin\left(\frac{2}{3}\sin\beta_0\right)\leq \beta_0\]
that is not possible for $\beta_0\in (0,\pi/2)$.

In case (ii), the angle differs from $\pi$ at most by $\arcsin\left(\frac{4\delta}{\lambda}\right)$, which is again too small for \eqref{angoli} to be true.

Since neither of cases (1) and (2) can be true, it is not possibile that $P\in R^\prime$, hence, $P\in \mathcal{U}_\delta\setminus R^\prime$ which implies that there is one of the endpoints of $l^\prime$, let us call it $P^\prime$ such that
\[dist(P,P^\prime)\leq \delta \sqrt{1+\frac{16}{\sin^2\beta_0}}.\]
\end{proof}
\begin{prop}\label{strozzavert}
Under the same assumptions of Theorem \ref{mainteo}, there exist positive constants $\ep_0$, $\alpha$ and $C>1$, depending only on the a priori data, such that, if
\[\ep:=\|\Lambda_{\gamma_0}-\Lambda_{\gamma_1}\|_*<\ep_0,\]
then $\mathcal{P}^0$ and $\mathcal{P}^1$ have the same number $N$ of vertices $\left\{P_j^0\right\}_{j=1}^N$ and $\left\{P_j^1\right\}_{j=1}^N$ respectively. Moreover, the vertices can be order so that
\begin{equation}\label{stimarozzavertici}
	dist\left(P_j^0,P_j^1\right)\leq \omega(\ep) \mbox{ for every }j=1,\ldots,N,
\end{equation}
where $\omega(\ep)=C \left|\log \ep\right|^{-\alpha^2/C}$.
\end{prop}
\begin{proof}It follows by the combination of Proposition \ref{strozza}, Lemma \ref{dadiffahaus} and Proposition \ref{propvertici}.\end{proof}
\subsection{Definition and differentiability of the function $F$}\label{sFdiff}

Let us denote by $\{P^j_i\}_{i=1}^N$ the vertices of polygon $\mathcal{P}^j$ for $j=0,1$ numbered in such a way that
$dist(P^0_i,P^1_i)\leq \omega(\ep)\mbox{ for }i=1,\ldots,N$,
for $\omega(\ep)$ as in Proposition \ref{strozzavert} and the segment $P^i_jP^i_{j+1}$ is a side of $\mathcal{P}^i$ for $i=0,1$ and $j=1,\ldots,N$.

Let us consider a deformation from $\mathcal{P}^0$ to $\mathcal{P}^1$: for $t\in[0,1]$ let
\[P_i^t=P^0_i+tv_i,\mbox{ where }v_i=P^1_i-P^0_i,\mbox{ for }i=1,\ldots,N\]
and denote by $\mathcal{P}^t$ the polygon  with vertices $P^t_j$ and sides  $P^t_jP^t_{j+1}$.

Let $\gamma_t=1+(k-1)\chi_{\mathcal{P}^t}$ and let $\Lambda_{\gamma_t}$ be the corresponding DtoN map.

As we proved in \cite[Corollary 4.5]{BFV17} the DtoN map $\Lambda_{\gamma_t}$ is differentiable with respect to $t$.

The function
\[F(t,\phi,\psi)=<\Lambda_{\gamma_t}(\phi), \psi>,\]
for $\phi,\psi\in H^{1/2}(\der\om)$,  is a differentiable function from $[0,1]$ to $\RR$ and we can write explicitly its derivative.

Let $u_t,v_t\in H^1(\om)$ be the solutions to
\[
\left\{\begin{array}{rcl}
             \dive(\gamma_t\nabla u_t) & = & 0\mbox{ in }\om, \\
             u_t&=&\phi\mbox{ on }\der\om,\\
           \end{array}
    \right.
\mbox{ and }
\left\{\begin{array}{rcl}
             \dive(\gamma_t\nabla v_t) & = & 0\mbox{ in }\om, \\
             v_t&=&\psi\mbox{ on }\der\om,\\
           \end{array}
    \right.
\]
and denote by $u_t^e$ and $v_t^e$ their the restrictions to $\om\setminus\mathcal{P}^t$ (and by $u_t^i$ and $v_t^i$ their restrictions to $\mathcal{P}^t$).

Let us fix an orthonormal system $(\tau_t, n_t)$ in such a way that $n_t$ represents almost everywhere the outward unit
normal to $\der \mathcal{P}_t$ and the tangent unit vector $\tau_t$ is oriented counterclockwise. Denote
by $M_t$ a $2\times 2$ symmetric matrix valued function defined on $\der\mathcal{P}_t$ with eigenvalues $1$ and $1/k$
and corresponding eigenvectors $\tau_t$ and $n_t$.

Let $\Phi_t^v$ be a map defined on $\der\mathcal{P}_t$, affine on each side of the polygon and such that
\[\Phi_t^v(P_i^t)=v_i\mbox{ for }i=1,\ldots,N.\]

Then, it was proved in \cite[Corollary 2.2]{BFV17} that, for all $t\in[0,1]$,
\[\frac{d}{dt}F(t,\phi,\psi)=(k-1)\int_{\der \mathcal{P}^t}M_t\nabla u_t^e\nabla v_t^e (\Phi_t^v\cdot n_t). \]

\subsection{Continuity at zero of the derivative of $F$}\label{sderivcont}
\begin{lm}\label{lcontder}
There exist constants $C$ and $\beta$, depending only on the a priori data, such that
\begin{equation}\label{contder}\left|\frac{d}{dt}F(t,\phi,\psi)-{\frac{d}{dt}F(t,\phi,\psi)}_{|_{t=0}}	\right|\leq C\|\phi\|_{H^{1/2}(\der\om)}\|\psi\|_{H^{1/2}(\der\om)}|v|^{1+\beta}t^\beta.\end{equation}
\end{lm}
\begin{proof}
This result corresponds to Lemma 4.4 in \cite{BFV17}. The dependence on $|v|$ is obtained by refining estimate (3.5) in \cite[Proposition 3.4]{BFV17} to get
\[\|u_t-u_0\|_{H^1(\om)}\leq C\|\phi\|_{H^{1/2}(\der\om)}\left|\mathcal{P}^t\Delta\mathcal{P}^0\right|^\theta\leq C_1\|\phi\|_{H^{1/2}(\der\om)}|v|^\theta t^\theta,\]
and by noticing that
\[\left|\Phi_t^v\right|\leq C|v|.\]
\end{proof}
\subsection{Bound from below for the derivative of $F$}\label{sboundbelow}
In this section we want to obtain a bound from below for the derivative of $F$ at $t=0$.
\begin{prop}\label{p3.3}
There exist a constant $m_1>0$, depending only on the a priori data, and a pair of functions $\tilde{\phi}$ and $\tilde{\psi}$ in $H^{1/2}(\der\om)$ such that
\begin{equation}\label{tesibasso}\left|{\frac{d}{dt}F(t,\tilde{\phi},\tilde{\psi})}_{|_{t=0}}\right|\geq m_1 |v| \|\tilde{\phi}\|_{H^{1/2}(\der\om)}\|\tilde{\psi}\|_{H^{1/2}(\der\om)}. \end{equation}
\end{prop}
\begin{proof}
Let us first normalize the length of vector $v$ and introduce
\[H(\phi,\psi)=\int_{\der\mathcal{P}_0} M_o\nabla u_0^e\nabla v_0^e\tilde{\Phi}_0^{v}\cdot n_0,\]
where
\[\tilde{\Phi}_0^{v}=\Phi_0^{v/|v|}.\]
By linearity, we have that ${\frac{d}{dt}F(t,\phi,\psi)}_{|_{t=0}}=|v|H(\phi,\psi)$.

Let $m_0=\|H\|_*=\sup\left\{\frac{H(\phi,\psi)}{\|\phi\|_{H^{1/2}(\der\om)}\|\psi\|_{H^{1/2}(\der\om)}}\,:\,\phi,\psi\neq 0\right\}$ be the operator norm of $H$, so that
\begin{equation}\label{normaH}\left|H(\phi,\psi)\right|\leq m_0\|\phi\|_{H^{1/2}(\der\om)}\|\psi\|_{H^{1/2}(\der\om)}\mbox{ for every }\phi,\psi\in H^{1/2}(\der\om).\end{equation}

Let $\Sigma$ be an open non empty subset of $\der\om$  and let us extend $\om$ to a open domain $\om_0=\om\cup D_0$  that
has Lipschitz boundary with constants $r_0/3$ and $K_0$ and such that $\Sigma$ is contained in $\om_0$ (see \cite{AV} for a detailed construction). Let us extend $\gamma_0$ by $1$ in $D_0$ (and still denote it by $\gamma_0$).

We denote by $G_0(x,y)$ the Green function corresponding to the operator $\dive(\gamma_0 \nabla\cdot)$ and to the domain $\om_0$.
The Green function $G_0(x,y)$ behaves like the fundamental solution of the Laplace equation $\Gamma(x,y)$ for points that are far from the polygon. For points close to the sides of the polygon but far from its vertices, the asymptotic behaviour of the Green function has been described in
\cite[Theorem 4.2]{AV} or \cite[Proposition 3.4]{BF11}: Let $y_r=Q+rn(y_0)$, where $Q$ is a point on $\der\mathcal{P}^0$ whose distance from the vertices of the polygons is greater than $r_0/4$ and $n(y_0)$ is the unit outer normal to $\der\mathcal{P}^0$. Then, for small $r$,
\begin{equation}\label{stimagreen}
	\left\|G_0(\cdot,y_r)-\frac{2}{k+1}\Gamma(\cdot,y_r)\right\|_{H^1(\om_0)}\leq C,
\end{equation}
where $C$ depends only on the a priori data.

Let us take $u_0=G_0(\cdot,y)$ and $v_0=G_0(\cdot,z)$ for $y,z\in K$, where $K$ is a compact subset of $D_0$ such that $dist(K,\der\om)\geq r_0/3$ and $K$ contains a ball of radius $r_0/3$. The functions $u_0$ and $v_0$ are both solutions to the equation $\dive(\gamma_0\nabla\cdot)=0$ in $\om$.

Define the function
\[S_0(y,z)=\int_{\der\mathcal{P}_0}M_0\nabla G_0(\cdot,y)\nabla G_0(\cdot,z)(\tilde{\Phi}_0^v\cdot n_0)\]
that, for fixed $z$, solves $\dive(\gamma_0 \nabla S_0(\cdot,z))=0$ in $\om\setminus\mathcal{P}^0$ and, for fixed $y$ it solves $\dive(\gamma_0 \nabla S_0(y,\cdot))=0$ in $\om\setminus\mathcal{P}^0$.

For $y,z\in K$, $S_0(y,z)=H(u_0,v_0)$, hence, by \eqref{normaH}
\begin{equation}\label{12.1}
|S_0(y,z)|\leq \frac{C_0m_0}{r_0^2}\mbox{ for }y,z\in K,
\end{equation}
where $C_0$ depend on the a priori data.

Moreover, by \eqref{stimagreen}, there exist $\rho_0$ and $E$ depending only on the a priori data such that
\begin{equation}
	\label{13.1}
	|S_0(y,z)|\leq E(d_yd_z)^{-1/2} \mbox{ for every } y,z\in \om\setminus\left(\mathcal{P}^0 \cup_{i=1}^NB_{\rho_0}(P_i^0)\right),
\end{equation}
where $d_y=dist(y,\mathcal{P}^0)$.

Since $S_0$ is small for $y,z\in K$ (see \eqref{12.1} and consider $m_0$ small), bounded for $y,z\in \om\setminus\mathcal{P}^0$ far from the vertices of the polygon, and since it is harmonic in $\om\setminus\mathcal{P}^0$, we can use a three balls inequality on a chain of balls in order to get a smallness estimate close to the sides of the polygon.

To be more specific, let $l_i$ be a side of $\mathcal{P}^0$ with endpoints $P^0_i$ and $P^0_{i+1}$. Let $Q^0_i$ be the midpoint of $l_i$ and let $y_r=Q^0_i+rn_i$ where $n_i$ is the unit outer normal to $\der\mathcal{P}^0$ at $Q^0_i$ and $r\in(0,K_0r_0)$.

\begin{lm}\label{small}
There exist constants $C>1$, $\beta$, and $r_1<r_0/C$ depending only on the a priori data, such that, for $r<r_1$
\begin{equation}
	\label{1.3.1}
	\left|S_0(y_r,y_r)\right|\leq C \left(\frac{\ep_0}{\ep_0+E}\right)^{\beta\tau^2_r}(\ep_0+E)r^{-1},
\end{equation}
where $\ep_0=m_0C_0r_0^{-2}$ and $\tau_r=\frac{1}{\log(1-r/r_1)}$.
\end{lm}
\begin{proof}For the proof of Lemma \ref{small} see \cite[Proposition 4.3]{BF11} where the estimate of $\tau_r$ is slightly more accurate.\end{proof}

Now, we want to estimate $\left|S_0(y_r,y_r)\right|$ from below. In order to accomplish this, let us take $\rho=\min\{d_0/4,r_0/4\}$ and write
\begin{align}\label{6.1}
\left|\frac{S_0(y_r,y_r)}{k-1}\right|\geq
&\left|\int_{\der\mathcal{P}_0\cap B_\rho(Q_i^0)}M_0\nabla G_0(\cdot,y_r)\nabla G_0(\cdot,y_r)(\tilde{\Phi}_0^v\cdot n_0)\right|\\&-\left|\int_{\der\mathcal{P}_0\setminus B_\rho(Q_i^0)}M_0\nabla G_0(\cdot,y_r)\nabla G_0(\cdot,y_r)(\tilde{\Phi}_0^v\cdot n_0)\right|
\\&:=I_1-I_2.\end{align}
The behaviour of the Green function (see \cite{AV}) gives immediately that, for $r<\rho/2$,
\begin{equation}\label{6.2}
I_2\leq C_1,
\end{equation}
for some $C_1$ depending only on the a priori data.

In order to estimate $I_1$, we add and subtract $\Gamma(\cdot,y_r)$ to $G_0(\cdot,y_r)$, then by Young inequality, \eqref{stimagreen}, and by the properties of $M_0$, we get
\begin{equation}\label{I1}
I_1\geq C_2\left|\int_{\der\mathcal{P}_0\cap B_\rho(Q_i^0)}\left|\nabla \Gamma(\cdot,y_r)\right|^2(\tilde{\Phi}_0^v\cdot n_0^i)\right|-C_3,
\end{equation}
where $C_2$ and $C_3$ depend only on the a priori data.

By definition of $\tilde{\Phi}_0^v$ we have
\[\left|\tilde{\Phi}_0^v(x)-\tilde{\Phi}_0^v(Q_i^0)\right|\leq C_4|x-Q^0_i|,\]
so, by adding and subtracting $\Phi^v_0(Q_i^0)$ into the integral of \eqref{I1}, we can write
\begin{align*}
\left|\int_{\der\mathcal{P}_0\cap B_\rho(Q_i^0)}\left|\nabla \Gamma(\cdot,y_r)\right|^2(\tilde{\Phi}_0^v\cdot n_0^i)\right|\geq
& \,\overline{\alpha} \int_{\der\mathcal{P}_0\cap B_\rho(Q_i^0)}\left|\nabla \Gamma(\cdot,y_r)\right|^2\\&-
C_4\int_{\der\mathcal{P}_0\setminus B_\rho(Q_i^0)}\left|\nabla \Gamma(\cdot,y_r)\right|^2|x-Q^0_i|,
\end{align*}
where $\overline{\alpha}=|\tilde{\Phi}_0^v(Q_i^0)\cdot n_0^i|$.
By straightforward calculations one can see that
\begin{equation}\label{8.2}
	 \int_{\der\mathcal{P}_0\cap B_\rho(Q_i^0)}\left|\nabla \Gamma(\cdot,y_r)\right|^2\geq \frac{C_5}{r}
\end{equation}
and
\begin{equation}\label{9.1}
	\int_{\der\mathcal{P}_0\setminus B_\rho(Q_i^0)}\left|\nabla \Gamma(\cdot,y_r)\right|^2|x-Q^0_i| \leq C_6\left|\log (\rho/r)\right|.
\end{equation}
By putting together \eqref{6.1}, \eqref{6.2}, \eqref{8.2} and \eqref{9.1}, we get
\begin{equation}\label{9.1star}
	\left|S_0(y_r,y_r)\right|\geq  \frac{C_6 \overline{\alpha}}{r}-C_7\left|\log (\rho/r)\right|-C_8.\end{equation}
 By comparing \eqref{1.3.1} and \eqref{9.1star} we get
\begin{equation}\label{9.2}
C_6\overline{\alpha}\leq C \left(\frac{\ep_0}{\ep_0+E}\right)^{\beta\tau^2_r}(\ep_0+E) +C_7r|\log(\rho/r)|+C_8r. 	
\end{equation}
By an easy calculation one can see that $\beta\tau_r^2\geq r^2/C_9$, hence
\begin{equation}\label{9.3}
C_6\overline{\alpha}\leq C \left(\frac{\ep_0}{\ep_0+E}\right)^{r^2/C_9}(\ep_0+E) +C_{10}\sqrt{r}. 	
\end{equation}
By choosing $r=\left|\log\left(\frac{\ep_0}{\ep_0+E}\right)\right|^{-1/4}$ and recalling that $\ep_0=C_0m_0r_0^{-2}$ we have
\[|\tilde{\Phi}_0^v(Q_i^0)\cdot n_0^i|=\overline{\alpha}\leq \omega_0(m_0),\]
where $\omega_0(t)$ is an increasing concave function such that $\lim_{t\to 0^+}\omega_0(t)=0$.

This estimate can also be obtained for $\tilde{\Phi}_0^v(y)\cdot n_0^i$ for every $y\in B_{\rho}(Q_i^0)\cap l_i$. Since $\tilde{\Phi}_0^v$ is linear on the bounded side $l_i$,
\[|\tilde{\Phi}_0^v(y)\cdot n_0^i|\leq\omega_0(m_0)\quad\quad \mbox{for every }y\in l_i,\]
and, in particular
\begin{equation}\label{vec1}\left|\frac{v_i}{|v|}\cdot n_0^i\right|=|\tilde{\Phi}_0^v(P_i)\cdot n_0^i|\leq\omega_0(m_0)\end{equation}
Repeating the same argument on the adjacent side,  $l_{i+1}$, containing $P_i$  we obtain in particular that  
\begin{equation}\label{vec2}
\left|\frac{v_i}{|v|}\cdot n_0^{i+1}\right|=|\tilde{\Phi}_0^v(P_i)\cdot n_0^{i+1}|\leq\omega_0(m_0)\end{equation}
Then,  there exists a constant $C>0$ depending on the a priori constants only such that 
\[\left|\frac{v_i}{|v|}\right|\leq C\omega_0(m_0)\]
and since one can apply the same procedure on each side of the polygon we have 
\[\left|\frac{v_i}{|v|}\right|\leq C\omega_0(m_0)\mbox{ for }i=1,\ldots,N\]
that yields
\[1\leq NC\omega_0(m_0)\Rightarrow m_0\geq \omega_0^{-1}(1/CN).\]
By definition of the operator norm of $H$, there exist $\tilde{\phi}$ and $\tilde{\psi}$ in $H^{1/2}(\der\om)$ such that
\[|H(\tilde{\phi},\tilde{\psi})|\geq \frac{m_0}{2}\|\tilde{\phi}\|_{H^{1/2}(\der\om)}\|\tilde{\psi}\|_{H^{1/2}(\der\om)}\]
and \eqref{tesibasso} is true for $m_1=\frac{\omega_0^{-1}(1/CN)}{2}$. 

\end{proof}
\begin{rem}
Note that the lower bound for the derivative of $F$ in Proposition \ref{p3.3} holds  for functions $\tilde{\phi}$ and $\tilde{\psi}$ with compact support on an open portion of $\partial\Omega$.
\end{rem}

\subsection{Lipschitz stability estimate}
In this section we conclude the proof of Theorem \ref{mainteo}.
Let $\tilde{\phi}$ and $\tilde{\psi}$ the functions the satisfy \eqref{tesibasso} in Proposition \ref{p3.3}.

By \eqref{tesibasso} and by \eqref{contder} we have

\begin{eqnarray*}\label{fine}
	 \left|<\left(\Lambda_{\gamma_1}-\Lambda_{\gamma_0}\right)(\tilde{\phi}),\tilde{\psi}>\right|\!\!&=\!\!&\left|F(1,\tilde{\phi},\tilde{\psi})-F(0,\tilde{\phi},\tilde{\psi})\right|=\left|\int_0^1\frac{d}{dt}F(t,\tilde{\phi},\tilde{\psi})dt\right|\\
\!\!	&\geq\!\!& \frac{d}{dt}F(t,\tilde{\phi},\tilde{\psi})_{|_{t=0}} \!-\!\int_0^1\!\left| \frac{d}{dt}F(t,\tilde{\phi},\tilde{\psi})-\frac{d}{dt}F(t,\tilde{\phi},\tilde{\psi})_{|_{t=0}}\right|dt\\
	&\geq&\left(m_1-C|v|^\beta\right)|v|\|\tilde{\phi}\|_{H^{1/2}(\der\om)}\|\tilde{\psi}\|_{H^{1/2}(\der\om)},
\end{eqnarray*}
that implies
\begin{equation}\label{fine2}
	\ep=\|\Lambda_{\gamma_0}-\Lambda_{\gamma_1}\|_*\geq \left(m_1-C|v|^\beta\right)|v|.
\end{equation}
From \eqref{stimarozzavertici}, since
$|v|\leq N\max_jdist(P^0_j,P^1_j)$ it follows that there exists $\ep_0>0$ such that, if
\[\|\Lambda_{\gamma_0}-\Lambda_{\gamma_1}\|_*\leq\ep_0,\] then
\[\left(m_1-C|v|^\beta\right)\geq m_1/2\]
and
\[|v|\leq \frac{2}{m_1}\|\Lambda_{\gamma_0}-\Lambda_{\gamma_1}\|_*.\]
Finally, since 
\[d_{H}\left( \partial \mathcal{P}^{0},\partial \mathcal{P}^{1}\right)\leq C| v|\]
 the claim follows.
\qed

Finally, as a byproduct of Theorem \ref{mainteo} and of Proposition \ref{propvertici} we have the following

\begin{coro}
Let $\mathcal{P}^0,\mathcal{P}^1\in\inspol$ and let
\[\gamma_0=1+(k-1)\chi_{\mathcal{P}^0}\mbox{ and }\gamma_1=1+(k-1)\chi_{\mathcal{P}^1}.\]
There exist $\ep_0$ and $C$ depending only on the a priori data such that, if
\[\|\Lambda_{\gamma_0}-\Lambda_{\gamma_1}\|_*\leq \ep_0,\]
then $\mathcal{P}^0$ and $\mathcal{P}^1$ have the same number $N$ of vertices $\left\{P_j^0\right\}_{j=1}^N$ and $\left\{P_j^1\right\}_{j=1}^N$ respectively. Moreover, the vertices can be ordered so that
\begin{equation}
dist\left(P_j^0,P_j^1\right)\leq C
\|\Lambda_{\gamma_0}-\Lambda_{\gamma_1}\|_*\quad \mbox{ for every }j=1,\ldots,N.
\end{equation}
\end{coro}

\section{Final remarks and extensions}
We have derived Lipschitz stability estimates for polygonal conductivity inclusions in terms of the Dirichlet-to-Neumann map using differentiability properties of the Dirichlet-to-Neuman map. \\
The result extends also to the case where finitely many conductivity polygonal inclusions are contained in the domain $\Omega$ assuming that they are at controlled distance one from the other and from the boundary of $\Omega$.\\
We expect that the same result holds also when having at disposal local data. In fact, as we observed at the end of Proposition 3.6, the lower bound for the derivative of $F$ is obtained using solutions with compact support in a open subset of $\partial\Omega$ and a rough stability estimate of the Hausdorff distance of polygons in terms of the local Dirichlet-to-Neumann map could be easily derived following the ideas contained in \cite{MR}.\\
Finally, it is relevant for the geophysical application we have in mind to extend the results of stability and reconstruction to the 3-D setting possibly considering  an inhomogeneous and/or anisotropic medium. This case is not at all straightforward since differentiability properties of the Dirichlet-to-Neumann map in this case are not known.

\bigskip

\textbf{Acknowledgment}

The paper was partially supported by GNAMPA - INdAM.

\end{document}